\documentclass[a4paper,twoside]{amsart}

\usepackage{verbatim}
\usepackage[british]{babel}
\usepackage{hyperref}
\usepackage{amsmath}
\usepackage{amssymb}
\usepackage{amsfonts}
\usepackage{amsxtra}
\usepackage{amsthm}
\usepackage{url}
\usepackage{latexsym}

\theoremstyle{definition}
\newtheorem{theorem}{Theorem}[section]
\newtheorem{definition}[theorem]{Definition}
\newtheorem{lemma}[theorem]{Lemma}
\newtheorem{proposition}[theorem]{Proposition}
\newtheorem{corollary}[theorem]{Corollary}
\newtheorem{remark}[theorem]{Remark}

\newcommand\laatepsilon\varepsilon

\author{Tim de Laat}
\thanks{The author is supported by the Danish National Research Foundation through the Centre for Symmetry and Deformation.}
\address{Department of Mathematical Sciences, University of Copenhagen,
\newline Universitetsparken 5, DK-2100 Copenhagen \O, Denmark}
\email{tlaat@math.ku.dk}

\title[On the JCB Grothendieck Theorem]{On the Grothendieck Theorem for jointly completely bounded bilinear forms}

\begin{document}

\maketitle

\begin{abstract}We show how the proof of the Grothendieck Theorem for jointly completely bounded bilinear forms on $C^{\ast}$-algebras by Haagerup and Musat can be modified in such a way that the method of proof is essentially $C^{\ast}$-algebraic. To this purpose, we use Cuntz algebras rather than type $\mathrm{III}$ factors. Furthermore, we show that the best constant in Blecher's inequality is strictly greater than one.
\end{abstract}

\section{Introduction} \label{laatsec:introduction}
In \cite{laatgrothendieck}, Grothendieck proved his famous \emph{Fundamental Theorem on the metric theory of tensor products}.\index{Grothendieck Theorem} He also conjectured a noncommutative analogue of this theorem for bounded bilinear forms on $C^{\ast}$-algebras. This \emph{noncommutative Grothendieck Theorem} was proved by Pisier assuming a certain approximability condition on the bilinear form \cite{laatpgt}. The general case was proved by Haagerup \cite{laathgi}. Effros and Ruan conjectured a ``sharper'' analogue of this theorem for bilinear forms on $C^{\ast}$-algebras that are jointly completely bounded (rather than bounded) \cite{laateranatos}.\index{Effros-Ruan conjecture} More precisely, they conjectured the following result, with universal constant $K=1$.
\begin{theorem}[JCB Grothendieck Theorem] \label{laattheorem:jcbgt}
  Let $A,B$ be $C^{\ast}$-algebras, and let $u:A \times B \rightarrow \mathbb{C}$ be a jointly completely bounded bilinear form. Then there exist states $f_1,f_2$ on $A$ and $g_1,g_2$ on $B$ such that for all $a \in A$ and $b \in B$,
  \begin{equation} \nonumber
    |u(a,b)| \leq K\|u\|_{jcb} \left( f_1(aa^*)^{\frac{1}{2}}g_1(b^*b)^{\frac{1}{2}} + f_2(a^*a)^{\frac{1}{2}}g_2(bb^*)^{\frac{1}{2}} \right),
  \end{equation}
where $K$ is a constant.
\end{theorem}
We call this \emph{Grothendieck Theorem for jointly completely bounded bilinear forms on $C^{\ast}$-algebras} the \emph{JCB Grothendieck Theorem}.\index{Grothendieck Theorem!jointly completely bounded} It is often referred to as the \emph{Effros-Ruan conjecture}.

In \cite{laatpsgt}, Pisier and Shlyakhtenko proved a version of Theorem \ref{laattheorem:jcbgt} for exact operator spaces,\index{operator space} in which the constant $K$ depends on the exactness constants of the operator spaces. They also proved the conjecture for $C^{\ast}$-algebras, assuming that at least one of them is exact, with universal constant $K=2^{\frac{3}{2}}$.

Haagerup and Musat proved the general conjecture (for $C^{\ast}$-algebras), i.e., Theorem \ref{laattheorem:jcbgt}, with universal constant $K=1$ \cite{laathmeffrosruan}. They used certain type $\mathrm{III}$ factors in the proof. Since the conjecture itself is purely $C^{\ast}$-algebraic, it would be more satisfactory to have a proof that relies on $C^{\ast}$-algebras. In this note, we show how the proof of Haagerup and Musat can be modified in such a way that essentially only $C^{\ast}$-algebraic arguments are used. Indeed, in their proof, one tensors the $C^{\ast}$-algebras on which the bilinear form is defined with certain type $\mathrm{III}$ factors, whereas we show that it also works to tensor with certain simple nuclear $C^{\ast}$-algebras admitting KMS states instead. We then transform the problem back to the (classical) noncommutative Grothendieck Theorem, as was also done by Haagerup and Musat.

Recently, Regev and Vidick gave a more elementary proof of both the JCB Grothendieck Theorem for $C^{\ast}$-algebras and its version for exact operator spaces \cite{laatregevvidick}. Their proof makes use of methods from quantum information theory and has the advantage that the transformation of the problem to the (classical) noncommutative Grothendieck Theorem is more explicit and based on finite-dimensional techniques. Moreover, they obtain certain new quantitative estimates.

For an extensive overview of the different versions of the Grothendieck Theorem, as well as their proofs and several applications, we refer to \cite{laatpisiersurvey}.

This text is organized as follows. In Section \ref{laatsec:bilforms}, we recall two different notions of complete boundedness for bilinear forms on operator spaces. In Section \ref{laatsec:cuntzalgebras}, we recall some facts about Cuntz algebras and their KMS states. This is needed for the proof of the JCB Grothendieck Theorem, which is given in Section \ref{laatsec:jcbgt} (with a constant $K > 1$) by using (single) Cuntz algebras. We explain how to obtain $K=1$ in Section \ref{laatsec:bestconstant}. In Section \ref{laatsec:blecher}, we show that using a recent result by Haagerup and Musat on the best constant in the noncommutative little Grothendieck Theorem, we are able to improve the best constant in Blecher's inequality.

\section{Bilinear forms on operator spaces} \label{laatsec:bilforms}
Recall that an operator space $E$ is a closed linear subspace of $\mathcal{B}(H)$ for some Hilbert space $H$. For $n \geq 1$, the embedding $M_n(E) \subset M_n(\mathcal{B}(H)) \cong \mathcal{B}(H^n)$ gives rise to a norm $\|.\|_n$ on $M_n(E)$. In particular, $C^{\ast}$-algebras are operator spaces. A linear map $T:E \rightarrow F$ between operator spaces induces a linear map $T_{n}:M_{n}(E) \rightarrow M_{n}(F)$ for each $n \in \mathbb{N}$, defined by $T_n([x_{ij}])=[T(x_{ij})]$ for all $x = [x_{ij}] \in M_n(E)$. The map $T$ is called completely bounded \index{completely bounded map} if the completely bounded norm $\|T\|_{cb}:=\sup_{n \geq 1} \|T_n\|$ is finite. 

There are two common ways to define a notion of complete boundedness for bilinear forms on operator spaces. For the first one, we refer to \cite{laatchristensensinclair}. Let $E$ and $F$ be operator spaces contained in $C^{\ast}$-algebras $A$ and $B$, respectively, and let $u:E \times F \rightarrow \mathbb{C}$ be a bounded bilinear form. Let $u_{(n)}:M_n(E) \times M_n(F) \rightarrow M_n(\mathbb{C})$ be the map defined by $([a_{ij}],[b_{ij}]) \mapsto \left[ \sum_{k=1}^n u(a_{ik},b_{kj}) \right]$.
\begin{definition}
  The bilinear form $u$ is called \emph{completely bounded} \index{bilinear form!completely bounded} if
\begin{equation} \nonumber
  \|u\|_{cb}:=\sup_{n \geq 1} \|u_{(n)}\|
\end{equation}
is finite. We put $\|u\|_{cb}=\infty$ if $u$ is not completely bounded.
\end{definition}
Equivalently (see Section 3 of \cite{laathmeffrosruan} or the Introduction of \cite{laatpsgt}), $u$ is completely bounded if there exists a constant $C \geq 0$ and states $f$ on $A$ and $g$ on $B$ such that for all $a \in E$ and $b \in F$,
\begin{equation} \label{laateq:cbconstant}
  |u(a,b)| \leq Cf(aa^*)^{\frac{1}{2}}g(b^*b)^{\frac{1}{2}},
\end{equation}
and $\|u\|_{cb}$ is the smallest constant $C$ such that \eqref{laateq:cbconstant} holds.

For the second notion, we refer to \cite{laatblecherpaulsen}, \cite{laateranatos}. Let $E$ and $F$ be operator spaces contained in $C^{\ast}$-algebras $A$ and $B$, respectively, and let $u:E \times F \rightarrow \mathbb{C}$ be a bounded bilinear form. Then there exists a unique bounded linear operator $\tilde{u}:E \rightarrow F^*$ such that
\begin{equation} \nonumber
  u(a,b)=\langle \tilde{u}(a),b \rangle
\end{equation}
for all $a \in E$ and $b \in F$, where $\langle .,. \rangle$ denotes the pairing between $F$ and its dual.
\begin{definition}
  The bilinear form $u$ is called \emph{jointly completely bounded} \index{bilinear form!jointly completely bounded} if the map $\tilde{u}:E \rightarrow F^*$ is completely bounded, and we set
\begin{equation} \nonumber
  \|u\|_{jcb}:=\|\tilde{u}\|_{cb}.
\end{equation}
We put $\|u\|_{jcb}=\infty$ if $u$ is not jointly completely bounded.
\end{definition}
Equivalently, if we define maps $u_n:M_n(E) \otimes M_n(F) \rightarrow M_n(\mathbb{C}) \otimes M_n(\mathbb{C})$ by
\[
  u_n\left( \sum_{i=1}^k a_i \otimes c_i,\sum_{j=1}^l b_j \otimes d_j \right) = \sum_{i=1}^k\sum_{j=1}^l u(a_i,b_j)c_i \otimes d_j
\]
for $a_1,\ldots,a_k \in A$, $b_1,\ldots,b_l \in B$, and $c_1,\ldots,c_k,d_1,\ldots,d_l \in M_n(\mathbb{C})$, then we have $\|u\|_{jcb}=\sup_{n \geq 1} \|u_n\|$.

\section{KMS states on Cuntz algebras} \label{laatsec:cuntzalgebras}
For $2 \leq n < \infty$, let $\mathcal{O}_n$ denote the Cuntz algebra \index{Cuntz algebra} generated by $n$ isometries, as introduced by Cuntz in \cite{laatcuntzalgebra}, in which one of the main results is that the algebras $\mathcal{O}_n$ are simple. We now recall some results by Cuntz. If $\alpha=(\alpha_1,\ldots,\alpha_k)$ denotes a multi-index of length $k=l(\alpha)$, where $\alpha_j \in \{1,\ldots,n\}$ for all $j$, we write $S_{\alpha}=S_{\alpha_1}\ldots S_{\alpha_k}$, and we put $S_0=1$. It follows that for every nonzero word $M$ in $\{S_i\}_{i=1}^n \bigcup \{S_i^*\}_{i=1}^n$, there are unique multi-indices $\mu$ and $\nu$ such that $M=S_{\mu}S_{\nu}^*$.

For $k \geq 1$, let $\mathcal{F}_n^k$ be the $C^{\ast}$-algebra generated by $\{S_{\mu}S_{\nu}^* \mid l(\mu)=l(\nu)=k\}$, and let $\mathcal{F}_n^0=\mathbb{C}1$. It follows that $\mathcal{F}_n^k$ is $*$-isomorphic to $M_{n^k}(\mathbb{C})$, and, as a consequence, $\mathcal{F}_n^k \subset \mathcal{F}_n^{k+1}$. The $C^{\ast}$-algebra $\mathcal{F}_n$ generated by $\bigcup_{k=0}^\infty \mathcal{F}_n^k$ is a UHF-algebra of type $n^{\infty}$.

If we write $\mathcal{P}_n$ for the algebra generated algebraically by $S_1,\ldots,S_n$, $S_1^*,\ldots,S_n^*$, each element $A$ in $\mathcal{P}_n$ has a unique representation
\begin{equation} \nonumber
  A = \sum_{k=1}^N (S_1^*)^k A_{-k} + A_0 + \sum_{k=1}^N A_kS_1^k,
\end{equation}
where $N \in \mathbb{N}$ and $A_k \in \mathcal{P}_n \cap \mathcal{F}_n$. The maps $F_{n,k}:\mathcal{P}_n \rightarrow \mathcal{F}_n$ ($k \in \mathbb{Z}$) defined by $F_{n,k}(A)=A_k$ extend to norm-decreasing maps $F_{n,k}:\mathcal{O}_n \rightarrow \mathcal{F}_n$. It follows that $F_{n,0}$ is a conditional expectation.

The existence of a unique KMS state on each Cuntz algebra was proved by Olesen and Pedersen \cite{laatolesenpedersen}. Firstly, we give some background on $C^{\ast}$-dynamical systems.
\begin{definition}
  A $C^{\ast}$-dynamical system $(A,\mathbb{R},\rho)$ \index{$C^{\ast}$-dynamical system} consists of a $C^{\ast}$-algebra $A$ and a representation $\rho:\mathbb{R} \rightarrow \mathrm{Aut}(A)$, such that each map $t \mapsto \rho_t(a)$, $a \in A$, is norm continuous.
\end{definition}
$C^{\ast}$-dynamical systems can be defined in more general settings. In particular, one can replace $\mathbb{R}$ with arbitrary locally compact groups.

Let $A^a$ denote the dense $*$-subalgebra of $A$ consisting of analytic elements, i.e., $a \in A^a$ if the function $t \mapsto \rho_t(a)$ has a (necessarily unique) extension to an entire operator-valued function. This extension is implicitly used in the following definition.
\begin{definition} \label{laatdefinition:kmsstate}
  Let $(A,\mathbb{R},\rho)$ be a $C^{\ast}$-dynamical system. An invariant state $\phi$ on $A$, i.e., a state for which $\phi \circ \rho_t=\phi$ for all $t \in \mathbb{R}$, is a KMS state \index{KMS state} if
\begin{equation} \nonumber
 \phi(\rho_{t+i}(a)b)=\phi(b\rho_{t}(a))
\end{equation}
for all $a \in A^a$, $b \in A$ and $t \in \mathbb{R}$.
\end{definition}
This definition is similar to the one introduced by Takesaki (see \cite{laattakesakitomita}, Definition 13.1). It corresponds to $\phi$ being a $\beta$-KMS state for $\rho_{-t}$ with $\beta=1$ according to the conventions of \cite{laatbrattelirobinson2} and \cite{laatolesenpedersen}. In the latter, the following two results were proved (see Lemma 1 and Theorem 2 therein). We restate these results slightly according to the conventions of Definition \ref{laatdefinition:kmsstate}.
\begin{proposition}\emph{(Olesen-Pedersen)}
  For all $t \in \mathbb{R}$ and the generators $\{S_k\}_{k=1}^n$ of $\mathcal{O}_n$, define $\rho^n_t(S_k)=n^{it}S_k$. Then $\rho^n_t$ extends uniquely to a $*$-automorphism of $\mathcal{O}_n$ for every $t \in \mathbb{R}$ in such a way that $(\mathcal{O}_n,\mathbb{R},\rho^n)$ becomes a $C^{\ast}$-dymamical system. Moreover, $\mathcal{F}_n$ is the fixed-point algebra of $\rho^n$ in $\mathcal{O}_n$, and $\mathcal{P}_n \subset (\mathcal{O}_n)^a$.
\end{proposition}
Let $\tau_n=\otimes_{k=1}^{\infty} \frac{1}{n} \mathrm{Tr}$ denote the unique tracial state on $\mathcal{F}_n$.
\begin{proposition}\emph{(Olesen-Pedersen)}
 For $n \geq 2$, the $C^{\ast}$-dynamical system given by $(\mathcal{O}_n,\mathbb{R},\rho^n)$ has exactly one KMS state, namely $\phi_n=\tau_n \circ F_{n,0}$.
\end{proposition}
For a $C^{\ast}$-algebra $A$, let $\mathcal{U}(A)$ denote its unitary group. The following result was proved by Archbold \cite{laatarchboldcuntz}. It implies the \emph{Dixmier property} for $\mathcal{O}_n$.
\begin{proposition} \label{laatproposition:cuntzdixmierproperty} \emph{(Archbold)} 
For all $x \in \mathcal{O}_n$,
\begin{equation} \nonumber
  \phi_n(x)1_{\mathcal{O}_n} \in \overline{\mathrm{conv}\{uxu^* \,|\, u \in \mathcal{U}(\mathcal{F}_n)\}}^{\|.\|}.
\end{equation}
\end{proposition}
As a corollary, we obtain the following (well-known) fact (see also \cite{laatcuntzautomorphisms}).
\begin{corollary} \label{laatcorollary:relativecommutant}
The relative commutant of $\mathcal{F}_n$ in $\mathcal{O}_n$ is trivial, i.e.,
\begin{equation} \nonumber
  (\mathcal{F}_n)^{\prime} \cap \mathcal{O}_n = \mathbb{C}1.
\end{equation}
\end{corollary}
\begin{proof}
   Let $x \in (\mathcal{F}_n)^{\prime} \cap \mathcal{O}_n$. By Proposition \ref{laatproposition:cuntzdixmierproperty}, we know that for every $\laatepsilon > 0$, there exists a finite convex combination $\sum_{i=1}^m \lambda_i u_i x u_i^*$, where $u_i \in \mathcal{U}(\mathcal{F}_n)$, such that $\|\sum_{i=1}^m \lambda_i u_i x u_i^* - \phi_n(x)1_{\mathcal{O}_n}\| < \laatepsilon$. Since $x \in (\mathcal{F}_n)^{\prime} \cap \mathcal{O}_n$, we have $\sum_{i=1}^m \lambda_i u_i x u_i^* = \sum_{i=1}^m \lambda_i x u_i u_i^*=x$. Hence, $\|x - \phi_n(x)1_{\mathcal{O}_n}\| < \laatepsilon$. This implies that $x \in \mathbb{C}1$.
\end{proof}
Proposition \ref{laatproposition:cuntzdixmierproperty} can be extended to finite sets in $\mathcal{O}_n$, as described in the following lemma, by similar methods as in \cite{laatdixmiervonneumannalgebras}, Part III, Chapter 5. For an invertible element $v$ in a $C^{\ast}$-algebra $A$, we define $\mathrm{ad}(v)(x)=vxv^{-1}$ for all $x \in A$.
\begin{lemma} \label{laatlemma:cuntzdixmierpropertyfinitesets}
  Let $\{x_1,\ldots,x_k\}$ be a subset of $\mathcal{O}_n$, and let $\laatepsilon>0$. Then there exists a convex combination $\alpha$ of elements in $\{\mathrm{ad}(u) \mid u \in \mathcal{U}(\mathcal{F}_n)\}$ such that
\[
  \|\alpha(x_i)-\phi_n(x_i)1_{\mathcal{O}_n}\| < \laatepsilon \quad \textrm{for all } i=1,\ldots,k.
\]
  Moreover, there exists a net $\{\alpha_j\}_{j \in J} \subset \mathrm{conv}\{\mathrm{ad}(u) \mid u \in \mathcal{U}(\mathcal{F}_n)\}$ such that
\[
  \lim_{j} \|\alpha_j(x)-\phi_n(x)1_{\mathcal{O}_n}\|=0
\]
for all $x \in \mathcal{O}_n$.
\end{lemma}
\begin{proof}
Suppose that $\|\alpha^{\prime}(x_i)-\phi_n(x_i)1_{\mathcal{O}_n}\| < \laatepsilon$ for $i=1,\ldots,k-1$. By Proposition \ref{laatproposition:cuntzdixmierproperty}, we can find a convex combination $\tilde{\alpha}$ such that
\[
  \|\tilde{\alpha}(\alpha^{\prime}(x_k))-\phi_n(\alpha^{\prime}(x_k))1_{\mathcal{O}_n}\| < \laatepsilon.
\]
Note that $\phi_n(\alpha^{\prime}(x_k))=\phi_n(x_k)$ and $1_{\mathcal{O}_n}=\tilde{\alpha}(1_{\mathcal{O}_n})$. By the fact that $\|\tilde{\alpha}(x)\| \leq \|x\|$ for all $x \in \mathcal{O}_n$, we conclude that $\alpha = \tilde{\alpha} \circ \alpha^{\prime}$ satisfies $\|\alpha(x_i)-\phi_n(x_i)1_{\mathcal{O}_n}\| < \laatepsilon$ for $i=1,\ldots,k$.

Let $J$ denote the directed set consisting of pairs $(F,\eta)$, where $F$ is a finite subset of $\mathcal{O}_n$ and $\eta \in (0,1)$, with the ordering given by $(F_1,\eta_1) \preceq (F_2,\eta_2)$ if $F_1 \subset F_2$ and $\eta_1 \geq \eta_2$. By the first assertion, this gives rise to a net $\{\alpha_j\}_{j \in J}$ with the desired properties.
\end{proof}

\section{Proof of the JCB Grothendieck Theorem} \label{laatsec:jcbgt}
In this section, we explain the proof of the Grothendieck Theorem for jointly completely bounded bilinear forms on $C^{\ast}$-algebras. As mentioned in Section \ref{laatsec:introduction}, the proof is along the same lines as the proof by Haagerup and Musat, but we tensor with Cuntz algebras instead of type $\mathrm{III}$ factors.

Applying the GNS construction to the pair $(\mathcal{O}_n,\phi_n)$, we obtain a $\ast$-representation $\pi_n$ of $\mathcal{O}_n$ on the Hilbert space $H_{\pi_n}=L^2(\mathcal{O}_n,\phi_n)$, with cyclic vector $\xi_n$, such that $\phi_n(x)=\langle \pi_n(x)\xi_n,\xi_n \rangle_{H_{\pi_n}}$. We identify $\mathcal{O}_n$ with its GNS representation. Note that $\phi_n$ extends in a normal way to the von Neumann algebra $\mathcal{O}_n^{\prime\prime}$, which also acts on $H_{\pi_n}$. This normal extension is a KMS state for a $W^{\ast}$-dynamical system with $\mathcal{O}_n^{\prime\prime}$ as the underlying von Neumann algebra (see Corollary 5.3.4 of \cite{laatbrattelirobinson2}). The commutant $\mathcal{O}_n^{\prime}$ of $\mathcal{O}_n$ is also a von Neumann algebra, and using Tomita-Takesaki theory \index{Tomita-Takesaki theory} (see \cite{laatbrattelirobinson2}, \cite{laattakesakitomita}), we obtain, via the polar decomposition of the closure of the operator $Sx\xi_n=x^*\xi_n$, a conjugate-linear involution $J:H_{\pi_n} \rightarrow H_{\pi_n}$ satisfying $J \mathcal{O}_n J \subset \mathcal{O}_n^{\prime}$.
\begin{lemma} For $k \in \mathbb{Z}$, we have
\[
  \mathcal{O}_n^k := \{x \in \mathcal{O}_n \mid \rho^n_t(x)=n^{-ikt}x \forall t \in \mathbb{R}\} = \{x \in \mathcal{O}_n \mid \phi_n(xy)=n^{-k}\phi_n(yx) \forall y \in \mathcal{O}_n\}.
\]
\end{lemma}
The proof of this lemma is analogous to Lemma 1.6 of \cite{laattakesakistructure}. Note that $\mathcal{O}_n^0=\mathcal{F}_n$, and that for all $k \in \mathbb{Z}$, we have $\mathcal{O}_n^k \neq \{0\}$. 
\begin{lemma} \label{laatlemma:ck}
  For every $k \in \mathbb{Z}$, there exists a $c_k \in \mathcal{O}_n$ such that
\[
  \phi_n(c_k^*c_k)=n^{\frac{k}{2}}, \qquad \phi_n(c_kc_k^*)=n^{-\frac{k}{2}},
\]
and, moreover, $\langle c_kJc_kJ\xi_n,\xi_n \rangle=1$.
\end{lemma}
The proof is similar to the proof of Lemma 2.1 of \cite{laathmeffrosruan}.
\begin{proposition} \label{laatproposition:tensor}
Let $A,B$ be $C^{\ast}$-algebras, and let $u:A \times B \rightarrow \mathbb{C}$ be a jointly completely bounded bilinear form. There exists a bounded bilinear form $\hat{u}$ on $(A \otimes_{\min} \mathcal{O}_n) \times (B \otimes_{\min} J\mathcal{O}_nJ)$ given by
\begin{equation} \nonumber
  \hat{u}(a \otimes c,b \otimes d)=u(a,b)\langle cd\xi_n,\xi_n \rangle
\end{equation}
for all $a \in A$, $b \in B$, $c \in \mathcal{O}_n$ and $d \in J\mathcal{O}_nJ$. Moreover, $\|\hat{u}\|\leq\|u\|_{jcb}$.
\end{proposition}
The $C^{\ast}$-algebra $J\mathcal{O}_nJ$ is just a copy of $\mathcal{O}_n$. This result is analogous to Proposition 2.3 of \cite{laathmeffrosruan}, and the proof is the same. Note that in our case, we use $\|\sum_{i=1}^k c_id_i\|_{\mathcal{B}(L^2(\mathcal{O}_n,\phi_n))}=\|\sum_{i=1}^k c_i \otimes d_i\|_{\mathcal{O}_n \otimes_{min} J\mathcal{O}_nJ}$ for all $c_1,\ldots,c_k \in \mathcal{O}_n$ and $d_1,\ldots,d_k \in J\mathcal{O}_nJ$. This equality is elementary, since $\mathcal{O}_n$ is simple and nuclear. In the proof of Haagerup and Musat, one takes the tensor product of $A$ and a certain type $\mathrm{III}$ factor $M$ and the tensor product of $B$ with the commutant $M^{\prime}$ of $M$, respectively. Note that $J\mathcal{O}_nJ \subset \mathcal{O}_n^{\prime}$.

One can formulate analogues of Lemma 2.4, Lemma 2.5 and Proposition 2.6 of \cite{laathmeffrosruan}. They can be proved in the same way as there, and one explicitly needs the existence and properties of KMS states on the Cuntz algebras (see Section \ref{laatsec:cuntzalgebras}). The analogue of Proposition 2.6 gives the ``transformation'' of the JCB Grothendieck Theorem to the noncommutative Grothendieck Theorem for bounded bilinear forms.

Using Lemma 2.7 of \cite{laathmeffrosruan}, we arrive at the following conclusion, which is the analogue of \cite{laathmeffrosruan}, Proposition 2.8.
\begin{proposition}
Let $K(n)=\sqrt{(n^{\frac{1}{2}}+n^{-\frac{1}{2}}) \slash 2}$, and let $u:A \times B \rightarrow \mathbb{C}$ be a jointly completely bounded bilinear form on $C^{\ast}$-algebras $A,B$. Then there exist states $f^n_1,f^n_2$ on $A$ and $g^n_1,g^n_2$ on $B$ such that for all $a \in A$ and $b \in B$,
\[
  |u(a,b)| \leq K(n)\|u\|_{jcb}\left(f^n_1(aa^*)^{\frac{1}{2}}g^n_1(b^*b)^{\frac{1}{2}} + f^n_2(a^*a)^{\frac{1}{2}}g^n_2(bb^*)^{\frac{1}{2}}\right).
\]
\end{proposition}
The above proposition is the JCB Grothendieck Theorem. However, the (universal) constant and states depend on $n$. This is because the noncommutative Grothendieck Theorem gives states on $A \otimes_{min} \mathcal{O}_n$ and $B \otimes_{min} J\mathcal{O}_nJ$, which clearly depend on $n$, and these states are used to obtain the states on $A$ and $B$. The best constant we obtain in this way comes from the case $n=2$, which yields the constant $K(2)=\sqrt{(2^{\frac{1}{2}}+2^{-\frac{1}{2}})/2} \sim 1.03$.

\section{The best constant} \label{laatsec:bestconstant}
In order to get the best constant $K=1$, we consider the $C^{\ast}$-dynamical system $(A,\mathbb{R},\rho)$, with $A = \mathcal{O}_2 \otimes \mathcal{O}_3$ and $\rho_t=\rho^2_t \otimes \rho^3_t$. It is straightforward to check that it has a KMS state, namely $\phi=\phi_2 \otimes \phi_3$. It is easy to see that $\mathcal{F}=\mathcal{F}_2 \otimes \mathcal{F}_3$ is contained in the fixed point algebra. (Actually, it is equal to the fixed point algebra, but we do not need this.) These assertions follow by the fact that the algebraic tensor product of $\mathcal{O}_2$ and $\mathcal{O}_3$ is dense in $\mathcal{O}_2 \otimes \mathcal{O}_3$. Note that $\rho$ is not periodic.

Applying the GNS construction to the pair $(A,\phi)$, we obtain a $\ast$-representation $\pi$ of $A$ on the Hilbert space $H_{\pi}=L^2(A,\phi)$, with cyclic vector $\xi$, such that $\phi(x)=\langle \pi(x)\xi,\xi \rangle_{H_{\pi}}$. We identify $A$ with its GNS representation. Using Tomita-Takesaki theory, we obtain a conjugate-linear involution $J:H_{\pi} \rightarrow H_{\pi}$ satisfying $J A J \subset A^{\prime}$ (see also Section \ref{laatsec:jcbgt}).

It follows directly from Proposition \ref{laatproposition:cuntzdixmierproperty} that $\phi(x)1_A \in \overline{\mathrm{conv}\{uxu^* \,|\, u \in \mathcal{U}(\mathcal{F})\}}^{\|.\|}$ for all $x \in A$. Also, the analogue of Lemma \ref{laatlemma:cuntzdixmierpropertyfinitesets} follows in a similar way, as well as the fact that $\mathcal{F}^{\prime} \cap A = \mathbb{C}1$.

It is elementary to check that
\[
  A_{\lambda,k}:=\{x \in A \mid \rho_t(x)=\lambda^{ikt}x \forall t \in \mathbb{R}\}=\{x \in A \mid \phi(xy)=\lambda^{k}\phi(yx) \forall y \in \mathcal{O}_n\}.
\]
Let $\Lambda:=\{2^p3^q \mid p,q \in \mathbb{Z}\} \cap (0,1)$. For all $\lambda \in \Lambda$ and $k \in \mathbb{Z}$, we have $A_{\lambda,k} \neq \{0\}$. This leads, analogous to Lemma \ref{laatlemma:ck}, to the following result.
\begin{lemma}
  Let $\lambda \in \Lambda$. For every $k \in \mathbb{Z}$ there exists a $c_{\lambda,k} \in A$ such that
\[
  \phi(c_{\lambda,k}^*c_{\lambda,k})=\lambda^{-\frac{k}{2}}, \qquad \phi(c_{\lambda,k}c_{\lambda,k}^*)=\lambda^{\frac{k}{2}}
\]
and
\[
  \langle c_{\lambda,k}Jc_{\lambda,k}J\xi,\xi \rangle=1.
\]
\end{lemma}
In this way, by the analogues of Lemma 2.4, Lemma 2.5 and Proposition 2.6 of \cite{laathmeffrosruan}, we obtain the following result, which is the analogue of \cite{laathmeffrosruan}, Proposition 2.8.
\begin{proposition}
Let $\lambda \in \Lambda$, and let $C(\lambda)=\sqrt{(\lambda^{\frac{1}{2}}+\lambda^{-\frac{1}{2}}) \slash 2}$. Let $u:A \times B \rightarrow \mathbb{C}$ be a jointly completely bounded bilinear form. Then there exist states $f^{\lambda}_1,f^{\lambda}_2$ on $A$ and $g^{\lambda}_1,g^{\lambda}_2$ on $B$ such that for all $a \in A$ and $b \in B$,
\[
  |u(a,b)| \leq C(\lambda)\|u\|_{jcb}\left(f^{\lambda}_1(aa^*)^{\frac{1}{2}}g^{\lambda}_1(b^*b)^{\frac{1}{2}} + f^{\lambda}_2(a^*a)^{\frac{1}{2}}g^{\lambda}_2(bb^*)^{\frac{1}{2}}\right).
\]
\end{proposition}
Note that $C(\lambda) > 1$ for $\lambda \in \Lambda$. Let $(\lambda_n)_{n \in \mathbb{N}}$ be a sequence in $\Lambda$ converging to $1$. By the weak*-compactness of the unit balls $(A^*_{+})_1$ and $(B^*_{+})_1$ of $A^*_{+}$ and $B^*_{+}$, respectively, the Grothendieck Theorem for jointly completely bounded bilinear forms with $K=1$ follows in the same way as in the ``Proof of Theorem 1.1'' in \cite{laathmeffrosruan}.
\begin{remark} \label{laatrmk:isomorphism}
By Kirchberg's second ``Geneva Theorem'' (see \cite{laatkirchbergphillips} for a proof), we know that $\mathcal{O}_2 \otimes \mathcal{O}_3 \cong \mathcal{O}_2$. This implies that the best constant in Theorem $\ref{laattheorem:jcbgt}$ can also be obtained by tensoring with the single Cuntz algebra $\mathcal{O}_2$, but considered with a different action that defines the $C^{\ast}$-dynamical system. Since the explicit form of the isomorphism is not known, we cannot adjust the action accordingly.
\end{remark}

\section{A remark on Blecher's inequality} \label{laatsec:blecher}
In \cite{laatblecher}, Blecher stated a conjecture \index{Blecher's inequality} about the norm of elements in the algebraic tensor product of two $C^{\ast}$-algebras. Equivalently, the conjecture can be formulated as follows (see Conjecture $0.2^{\prime}$ of \cite{laatpsgt}). For a bilinear form $u:A \times B \rightarrow \mathbb{C}$, put $u^t(b,a)=u(a,b)$.
\begin{theorem}[Blecher's inequality] \label{laattheorem:blecher2}
  There is a constant $K$ such that any jointly completely bounded bilinear form $u:A \times B \rightarrow \mathbb{C}$ on $C^{\ast}$-algebras $A$ and $B$ decomposes as a sum $u=u_1+u_2$ of completely bounded bilinear forms on $A \times B$, and $\|u_1\|_{cb} + \|u_2^t\|_{cb} \leq K\|u\|_{jcb}$.
\end{theorem}
A version of this conjecture for exact operator spaces and a version for pairs of $C^{\ast}$-algebras, one of which is assumed to be exact, were proved by Pisier and Shlyakhtenko \cite{laatpsgt}. They also showed that the best constant in Theorem \ref{laattheorem:blecher2} is greater than or equal to $1$. Haagerup and Musat proved that Theorem \ref{laattheorem:blecher2} holds with $K=2$ \cite[Section 3]{laathmeffrosruan}. We show that the best constant is actually strictly greater than $1$.

In the following, let $\mathrm{OH}(I)$ denote Pisier's operator Hilbert space \index{operator Hilbert space} based on $\ell^2(I)$ for some index set $I$. Recall the noncommutative little Grothendieck Theorem. \index{Grothendieck Theorem!noncommutative little}
\begin{theorem}[Noncommutative little Grothendieck Theorem]
  Let $A$ be a $C^{\ast}$-algebra, and let $T:A \rightarrow \mathrm{OH}(I)$ be a completely bounded map. Then there exists a universal constant $C > 0$ and states $f_1$ and $f_2$ on $A$ such that for all $a \in A$,
\begin{equation} \nonumber
  \|Ta\| \leq C\|T\|_{cb}f_1(aa^*)^{\frac{1}{4}}f_2(a^*a)^{\frac{1}{4}}.
\end{equation}
\end{theorem}
For a completely bounded map $T:A \rightarrow \mathrm{OH}(I)$, denote by $C(T)$ the smallest constant $C > 0$ for which there exist states $f_1$, $f_2$ on $A$ such that for all $a \in A$, we have $\|Ta\| \leq Cf_1(aa^∗)^{\frac{1}{4}}f_2(a^*a∗)^{\frac{1}{4}}$. In \cite{laathmeffrosruan}, Haagerup and Musat proved that $C(T) \leq \sqrt{2}\|T\|_{cb}$. Pisier and Shlyakhtenko proved in \cite{laatpsgt} that $\|T\|_{cb} \leq C(T)$ for all $T:A \rightarrow \mathrm{OH}(I)$. Haagerup and Musat proved that for a certain $T:M_3(\mathbb{C}) \rightarrow \mathrm{OH}(3)$, the inequality is actually strict, i.e., $\|T\|_{cb} < C(T)$ \cite[Section 7]{laathmfactorization}. We can now apply this knowledge to improve the best constant in Theorem \ref{laattheorem:blecher2}.
\begin{theorem}
  The best constant $K$ in Theorem \ref{laattheorem:blecher2} is strictly greater than $1$.
\end{theorem}
\begin{proof}
  Let $A$ be a $C^{\ast}$-algebra, and let $T:A \rightarrow \mathrm{OH}(I)$ be a completely bounded map for which $\|T\|_{cb} < C(T)$. Define the map $V=\overline{T^*}JT$ from $A$ to $\overline{A^*}=\overline{A}^*$, where $J:\mathrm{OH}(I) \rightarrow \overline{\mathrm{OH}(I)^*}$ is the canonical complete isomorphism and $T^*:\mathrm{OH}(I)^* \rightarrow A^*$ is the adjoint of $T$. Hence, $V$ is completely bounded. It follows that $V=\tilde{u}$ for some jointly completely bounded bilinear form $u:A \times \overline{A} \rightarrow \mathbb{C}$. Moreover, $\|u\|_{jcb}=\|V\|_{cb}=\|T\|_{cb}^2$, where the last equality follows from the proof of Corollary 3.4 in \cite{laatpsgt}. By Blecher's inequality, i.e., Theorem \ref{laattheorem:blecher2}, we have a decomposition $u=u_1+u_2$ such that $\|u_1\|_{cb}+\|u_2^t\|_{cb} \leq K\|u\|_{jcb}$.

By the second characterization of completely bounded bilinear forms (in the Christensen-Sinclair sense) in Section \ref{laatsec:bilforms}, we obtain
\begin{equation} \nonumber
  |u_1(a,b)| \leq \|u_1\|_{cb} f_1(aa^*)^{\frac{1}{2}}g_1(b^*b)^{\frac{1}{2}}, \quad |u_2(a,b)| \leq \|u_2^t\|_{cb} f_2(a^*a)^{\frac{1}{2}}g_2(bb^*)^{\frac{1}{2}}.
\end{equation}
It follows that
\begin{equation} \nonumber
  |u(a,b)| \leq \|u_1\|_{cb}f_1(aa^*)^{\frac{1}{2}}g_1(b^*b)^{\frac{1}{2}} + \|u_2^t\|_{cb}f_2(a^*a)^{\frac{1}{2}}g_2(bb^*)^{\frac{1}{2}}.
\end{equation}
Let $\overline{g}_i(a)=g_i(\overline{a^*})$ for $i=1,2$, and define states
\[
  \tilde{f}=\frac{\|u_1\|_{cb}f_1+\|u_2^t\|_{cb}\overline{g}_2}{\|u_1\|_{cb}+\|u_2^t\|_{cb}} \;\textrm{ and }\; \tilde{g}=\frac{\|u_1\|_{cb}\overline{g}_1+\|u_2^t\|_{cb}f_2}{\|u_1\|_{cb}+\|u_2^t\|_{cb}}.
\]
We obtain
\begin{equation} \nonumber
\begin{split}
  &\|T(a)\|^2 = |u(a,\overline{a})| \leq \|u_1\|_{cb}f_1(aa^*)^{\frac{1}{2}}\overline{g}_1(a^*a)^{\frac{1}{2}} + \|u_2^t\|_{cb}f_2(a^*a)^{\frac{1}{2}}\overline{g}_2(aa^*)^{\frac{1}{2}} \\
  & \; \leq (\|u_1\|_{cb}f_1+\|u_2^t\|_{cb}\overline{g}_2)(aa^*)^{\frac{1}{2}}(\|u_1\|_{cb}\overline{g}_1+\|u_2^t\|_{cb}f_2)(a^*a)^{\frac{1}{2}} \\
  & \; \leq (\|u_1\|_{cb}+\|u_2^t\|_{cb})\tilde{f}(aa^*)^{\frac{1}{2}}\tilde{g}(a^*a)^{\frac{1}{2}}.
\end{split}
\end{equation}
Hence, $\|u_1\|_{cb}+\|u_2^t\|_{cb} \geq C(T)^2 > \|T\|_{cb}^2 = \|u\|_{jcb}$. This proves the theorem.
\end{proof}

\section*{Acknowledgements}
The question if a $C^{\ast}$-algebraic proof of Theorem \ref{laattheorem:jcbgt} exists was suggested to me by Uffe Haagerup and Magdalena Musat. I thank them for many useful comments.

\end{document}